\documentclass[paper=a4, fontsize=11pt]{scrartcl} % A4 paper and 11pt font size

\usepackage[T1]{fontenc} % Use 8-bit encoding that has 256 glyphs
\usepackage{fourier} % Use the Adobe Utopia font for the document - comment this line to return to the LaTeX default
\usepackage[english]{babel} % English language/hyphenation
\usepackage{amsmath,amsfonts,amsthm} % Math packages
\usepackage{amsmath, calligra, mathrsfs}
\usepackage{amssymb}
\usepackage[mathscr]{euscript}
\usepackage{verbatim}
\usepackage[dvipsnames]{xcolor}
\usepackage{mdframed} % documentation "The mdframed package" on https://www.ps.uni-saarland.de/Publications/documents/Aiken_2009_File.pdf
\usepackage{soulutf8}
\usepackage{stmaryrd}
\usepackage{tikz-cd}
\usepackage{hyperref}
\usepackage{lipsum} % Used for inserting dummy 'Lorem ipsum' text into the template

\usepackage{sectsty} % Allows customizing section commands
\allsectionsfont{\centering \normalfont\scshape} % Make all sections centered, the default font and small caps

\usepackage{fancyhdr} % Custom headers and footers
\usepackage{xurl}
\newcommand*{\sheafhom}{\mathcal{H}\kern -.5pt om}
\numberwithin{equation}{section} % Number equations within sections (i.e. 1.1, 1.2, 2.1, 2.2 instead of 1, 2, 3, 4)
\numberwithin{figure}{section} % Number figures within sections (i.e. 1.1, 1.2, 2.1, 2.2 instead of 1, 2, 3, 4)
\numberwithin{table}{section} % Number tables within sections (i.e. 1.1, 1.2, 2.1, 2.2 instead of 1, 2, 3, 4)

\newtheorem{thm}{Theorem}[section]
\newtheorem{cor}[thm]{Corollary}
\newtheorem{prop}[thm]{Proposition}

\theoremstyle{definition}
\newtheorem{defn}[thm]{Definition}

\newtheorem{exmp}[thm]{Example}

\theoremstyle{remark}
\newtheorem{rem}[thm]{Remark}

\DeclareMathOperator{\rk}{rank}

\newcommand{\horrule}[1]{\rule{\linewidth}{#1}} % Create horizontal rule command with 1 argument of height

\title{	
	\normalfont \normalsize 
	\textsc{} \\ [25pt] % Your university, school and/or department name(s)
	\horrule{0.5pt} \\[0.4cm] % Thin top horizontal rule
	\huge Matroidal Cayley-Bacharach and independence/dependence of geometric properties of matroids

	\horrule{2pt} \\[0.5cm] % Thick bottom horizontal rule
}

\author{Soohyun Park} % Your name

\date{\normalsize September 16, 2022} % Today's date or a custom date

\begin{document}
	
	\maketitle 
	
	\begin{abstract}
	 \noindent We consider the relationship between a matroidal analogue of the degree $a$ Cayley-Bacharach property (finite sets of points failing to impose independent conditions on degree $a$ hypersurfaces) and geometric properties of matroids. If the matroid polytopes in question are nestohedra, we show that the minimal degree matroidal Cayley-Bacharach property denoted $MCB(a)$ is determined by the structure of the building sets used to construct them. This analysis also applies for other degrees $a$. Also, it does not seem to affect the combinatorial equivalence class of the matroid polytope.  \\
	 
	 \noindent However, there are close connections to minimal nontrivial degrees $a$ and the geometry of the matroids in question for paving matroids (which are conjecturally generic among matroids of a given rank) and matroids constructed out of supersolvable hyperplane arrangements. The case of paving matroids is still related to with properties of building sets since it is closely connected to (Hilbert series of) Chow rings of matroids, which are combinatorial models of the cohomology of wonderful compactifications. Finally, our analysis of supersolvable line and hyperplane arrangements give a family of matroids which are natrually related to independence conditions imposed by points one plane curves or can be analyzed recursively.
	\end{abstract}
	
	\section{Introduction}
	
	Motivated by recent rationality-related results, Levinson and Ullery \cite{LU} recently defined the degree $r$ Cayley--Bacharach property $CB(r)$ of finite sets $\Gamma \subset \mathbb{P}^n$ to mean ones that fail to impose independent conditions on the space of degree $r$ homogeneous polynomials. For a family of cases, they show that such a set $\Gamma$ lies on a union of low-dimensional linear subspaces )Theorem 1.3 on p. 2 of \cite{LU}. In Question 7.6 on p. 14 of \cite{LU}, they asked whether a matroidal analogue of their result holds. In \cite{Pa}, we show that this does not hold (Theorem 1.6 and Theorem 1.8 on p. 4 of \cite{Pa}) and explore combinatorial criteria for $MCB(a)$ to hold. \\

	We consider different directions where the matroidal Cayley-Bacharach condition $MCB(a)$ is independent of or dependent on the geometry of the matroids involved or the objects they are constructed from. For minimal $a$, we analyze how $MCB(a)$ relates to properties of building sets used to construct the nestohedron (Theorem \ref{nestmcb}). However, this result also shows that there the $MCB(a)$ property does not measure a form of combinatorial equivalence for matroid polytopes which form nestohedra. We can consider what happens in a (conjecturally) generic setting among matroids of a given rank using paving matroids. In this setting, we study cases where the minimal degree $a$ where $MCB(a)$ is satisfied nontrivial is small (Theorem \ref{pavhyperplrat}) and show that lowering such $a$ correlates to larger degree terms in the Hilbert series of the Chow ring of the matroid (Corollary \ref{chowpav}), which is a combinatorial model for the cohomology of wonderful compactifications. \\
	
	Finally, we use supersolvable arrangements of linear subspaces to find compare minimal degrees $a$ where $MCB(a)$ can be satisfied with degrees $D$ where a collection of points fail to impose independent conditions on plane curves of degree $D$ (Proposition \ref{supersolmcbexp}) and a recursive argument for the $MCB(a)$ property on supersolvable hyperplane arrangements (Proposition \ref{supersolmcb}). Note that the case of line arrangements gives a family of matroids other than the representable case which naturally parametrizes questions about independence of conditions imposed by points on hypersurfaces. This result also shows that $MCB(a)$ properties of supersolvable hyperplane arrangements can be analyzed recursively and that the flats satisfy special clustering properties.
	
	\section*{Acknowledgements}
	I am very grateful to Benson Farb for his guidance and encouragement. Also, I would like to thank Laura Anderson and Rephael Wenger for some clarifications on definitions used in some references.
	
	\section{Independence from geometry and nestohedra}

	Based on the matroidal Cayley-Bacharach property $MCB(a)$ of degree $a$, we define $MCB(a)$ for a building set $B$ (Definition 7.1 on p. 1044 of \cite{Pos}).The motivation/connection to the ``usual'' $MCB(a)$ property (Question 7.6 on p. 14 of \cite{LU}) comes from the fact that the half-space description of matroid polytopes is determined by flats of the matroid (Proposition 2.3 on p. 441 of \cite{FS}).  \\

	\begin{defn} (Levinson--Ullery, p. 14 of \cite{LU}, p. 2 of \cite{Pa} ) \\
		A matroid $M$ with underlying set $E$ satisfies the \textbf{matroidal Cayley-Bacharach property of degree $a$} if $\bigcup_{i = 1}^a F_i \supset E \setminus p \Longrightarrow \bigcup_{i  = 1}^a F_i = E$ for any $p \in E$ and flats $F_1, \ldots, F_a$ of $M$.
	\end{defn}

	\begin{defn}
		Let $[n] = \{ 1. \ldots, n \}$ A building set $B$ (Definition 7.1 on p. 1044 of \cite{Pos}) \textbf{satisfies $\mathbf{MCB(a)}$} if $\bigcup_{i = 1}^a I_i \supset [n] \setminus \{ k \} \Longrightarrow \bigcup_{i = 1}^a I_i = [n]$ for all $k \in [n]$. 
	\end{defn}
	
	 In the case of nestohedra constructed out of connected building sets containing the ground set $[n]$, this is identical to the original matroidal Cayley--Bacharch property since the facets correspond to maximal elements of $B \setminus [n]$ (Proposition 3.12, Corollary 3.13, and Theorem 3.14 on p. 451 -- 452 of \cite{FS}). \\ 
	
	We can use this to show that the $MCB(a)$ is ``independent'' of combinatorial equivalence properties of matroids whose polytopes are nestohedra.
	
	\begin{defn} (p. 450 of \cite{FS} Definition 7.1 on p. 1044 of \cite{Pos}, Proposition 7.5 on p. 1046 of \cite{Pos}) \\
		\begin{enumerate}
			\item Given a family $\mathcal{F}$ of subsets of $[n]$, we associated the following Minkowski sum of simplices \[ \Delta_{\mathcal{F}} = \sum_{F \in \mathcal{F}} \Delta_F. \]
			
			\item A collection $B$ of nonempty subsets in $S$ is a \textbf{building set on $S$} if it satisfies the following conditions:
			
				\begin{itemize}
					\item If $I, J \in B$ and $I \cap J \ne \emptyset$, then $I \cup J \in B$.
					
					\item $B$ contains all singletons $\{ i \}$ for $i \in S$.
				\end{itemize}
			
			\item A \textbf{nestohedron} is a polytope from Part 1 where $\mathcal{F}$ is a building set.
		\end{enumerate}
	\end{defn}

	\begin{thm} \label{nestmcb} ~\\
		\vspace{-3mm}
		\begin{enumerate}
			\item The minimal degree $a$ such that a nestohedron $P = \sum_{I \in B} \Delta_I$ constructed from a connected building set $B$ on $[n] = \{ 1, \ldots, n \}$ can satisfy $MCB(a)$ nontrivially is given by $n - \dim P$. This is satisfied if and only if each maximal element $I \in B_{\max} \subset B \setminus [n] \subset 2^{[n]}$ has $\ge 2$ subsets which are maximal among those contained in $I$. Finally, it is the only degree where this is possible.

			\item In Part 1, the degree $a$ is given by $n - c$, where $c$ is the number of connected components of the nestohedron built out of $B \setminus [n]$. If this latter polytope is a matroid polytope $P_M$ for some matroid $M$, it is also the equal to $n - c(M)$, where $c(M)$ is the number of connected components of $M$. 
			
			\item For nestonedra, the $MCB(a)$ property is not a combinatorial invariant. In other words, there are matroids which yield combinatorially equivalent matroid polytopes where one satisfies $MCB(a)$ for some $a$ and the other does not.
		\end{enumerate}
		
	\end{thm}

	\begin{rem}
		By Lemma 3.10 on p. 450 of \cite{FS}, \emph{any} collection of subsets of $[n]$ has a unique minimal extension which is a building set called the \textbf{building set closure}. The statements above and arguments used below can be repeated with the building sets replaced by building set closures.
	\end{rem}
	
	\color{black}
	
	\begin{proof}
		\begin{enumerate}
			% Check p. 1048 of Postnikov \cite{Pos} (below condition (F3)) for condition...
			\item There is a correspondence between nested sets $N \subset B \subset 2^{[n]}$ and faces of a generalized permutohedron (Proposition 7.5 on p. 1046 of \cite{Pos}, proof of Proposition 7.9 on p. 1048 of \cite{Pos}, Proposition 3.12, Corollary 3.13, and Theorem 3.14 on p. 451 -- 452 of \cite{FS}). In this correspondence, the facets are parametrized by elements $I \in N \subset 2^{[n]}$ for nested sets $N$. This is the comes from the same reasoning which shows that flats give the half-space description of a matroid polytope $P_M$. Next, we use the fact that maximal nested sets correspond to $B$-forests (Proposition 7.8 on p. 1048 of \cite{Pos}). Note that any subcollection of a nested set $N \subset 2^{[n]}$ containing the elements of $B_{\max}$ (i.e. maximal elements) is still a nested set and that maximal subcollections is still a nested set. Also, any nested set is contained in a unique maximal building set. \\
			
			The number of minimal nested subsets (i.e. $T_{\le i}$ for nodes $i$) show that the unions of smaller nested sets are missing $\ge 1$ subset for each case if and only if there are $\ge 2$ ``almost maximal'' subsets. Finally, we use the fact that $\dim P = n - |B_{\max}|$. The degree $a$ is the only one allowed since allowing larger degrees would include cases where two ``almost maximal'' building sets are used in place of a single maximal building set (a degree $a + 1$ case).
			
			\item This is an application of Proposition 2.4 on p. 442 and Remark 3.11 on p. 450 of \cite{FS}.

			\item Using Part 1, we see that $MCB(a)$ depends on the number of maximal elements in the building set.  Note that $MCB(a)$ cannot be satisfied when the facets do not come from maximal elements of the building sets since they come from those of the a dilation of standard simplex on $\mathbb{R}^n$. However, \emph{any} nestohedron is combinatorially equivalent one from a connected building set even if the facets do not necessarily come from maximal elements of the building sets (Corollary 5 on p. 189 of \cite{V}, p. 122 of \cite{E}).  Then, the situation in Part 1 applies. For degrees $a$ less than the number of building sets maximal among those excluding $[m]$ (the ground set of the new connected building set), the $MCB(a)$ property is trivially satisfied. However, this is not true for the original one even when $a = 1$. Both cases come from combinatorially equivalent polytopes. 
		\end{enumerate}
	\end{proof}

	\section{Geometry determined by $MCB(a)$}

	\subsection{Paving matroids}
	
	We will focus on the case of paving matroids (which are conjectured to make up almost all matroids while known logarithmically \cite{Pv}) and their geometric structure. More specifically, we will explore connections to the Chow rings of these matroids. Note that these rings are still connected to properties of building sets since they are a combinatorial model for the cohomology of wonderful compactifications, which are built out of building sets. \\ 
	
	In the case of a paving matroid of rank $m + 1$ with ground set $E = [n] = \{ 1, \ldots, n \}$, the hyperplanes are given by $m$-partitions of $[n]$ (Proposition 2.1.24 on p. 71 of \cite{Ox}). These are collections of subsets of $E$ such that any $m$-element subset of $E$ is contained in a unique member of this collection. If we take $F_1, \ldots, F_a$ in the definition of $MCB(a)$ to be \emph{any} collection of flats (possibly with repeats), we need to consider \emph{minimal} numbers of flats which cover all but possibly one element of $E$. \\

	This means that we look at the unions of the smallest number of hyperplanes covering $E$. Depending on how small the degree $a$ is and how much the sizes of the hyperplanes varies, it may be easier or more difficult to build a paving matroid of rank $m + 1$ satisfying $MCB(a)$. For example, setting a large lower bound for hyperplanes which do \emph{not} have size $m$ implies that a lower degree $a$ can be satisfied since a larger gap is left behind by removing a large hyperplane. Note that \emph{any} collection of covering sets to be used in the definition of $MCB(a)$ has an associated paving matroid $M$ of rank $m + 1$ with these ``covering sets'' $F_i$ as a subcollection of the hyperplanes of $M$ (Proposition \ref{pavhyperplrat}).  Also, this corresponds to lengths of chains/number of terms used and the sizes of the coefficients in the Hilbert function of the Chow ring of the matroid (Corollary 2 and comments on p. 525 -- 526 of \cite{FY}) (Corollary \ref{chowpav}).  Generalizations to arbitrary matroids satisfying $MCB(a)$ are outlined in Remark \ref{degintgen}.

	\begin{prop}
		If a paving matroid $M$ with ground set $E = [n] = \{ 1, \ldots n \}$ satisfies $MCB(b)$ for some $b$, then it must satisfy $MCB(a)$ with $a$ equal to the smallest number of hyperplanes that can cover $E$. 
	\end{prop}
	
	\begin{proof}
		Since we are allowed to repeat flats, any matroid failing to satisfy $MCB(a)$ will \emph{not} satisfy $MCB(b)$ for any $b \ge a$. The minimal possible degree where this occurs is the smallest number of hyperplanes that can cover $E$.
	\end{proof}
	
	\begin{thm} \label{pavhyperplrat} ~\\
		\vspace{-5mm}
		
		\begin{enumerate}
			% Removed part about every m-element subset lying in at most one H_i
			% See statement about m-partitions in Proposition 2.1.24 on p. 71 of Oxley \cite{Ox}
			\item Let $M$ be a paving matroid of rank $m + 1$ such that the largest $k$ hyperplanes $H_1, \ldots, H_k$ of $M$ form a cover of the ground set $E = [n] = \{ 1, \ldots, n \}$. Consider a family of such paving matroids. 
			
			Suppose that there is a constant $C \in \mathbb{Z}_{> 0}$ such that $\frac{\max |H_i|}{\min |H_i|} < C$. If $\frac{n}{Ck^2 (m - 1)} >> k$ (i.e. $\frac{n}{Ck^3 (m - 1)} \to \infty$ as $m \to \infty$ treating the variables as functions of $m$), then $M$ satisfies $MCB(a)$ for $a \le k - 1 + \frac{n}{2Ck^2 (m - 1)}$. In general, this is true whenever $\min |H_i| >> k(m - 1)$. \\

			\item Consider paving matroids $M$ such that the largest $k$ hyperplanes $H_1, \ldots, H_k$ form a cover of the ground set $[n]$ as in Part 1.  If $a < 1 + (k - 1) \frac{\min |H_i|}{k(m - 1)}$, then $M$ satisfies $MCB(a)$. 
		\end{enumerate} 
	\end{thm}
	
	\begin{proof}
		\begin{enumerate}

			\item The assumption that $\frac{\max |H_i|}{\min |H_i|} < C$ implies that the sizes of the $H_i$ don't vary much. We would like to use the bound $\frac{n}{Ck^2 (m - 1)} >> k$ to show that the only flats to consider for the $MCB(a)$ condition are the maximal hyperplanes $H_i$ covering the ground set $[n]$. \\

			Note that $\max |H_i| \ge \frac{n}{k}$ since $H_1, \ldots, H_k$ cover $[n]$,. Since $\frac{\max |H_i|}{\min |H_i|} < C$, we have that $\min |H_i| > \frac{\max |H_i|}{C} \ge \frac{n}{Ck}$. We can use this to study covers of $[n]$ by hyperplanes of the paving matroid. Given that any hyperplane is contained in a maximal hyperplane, the size of the next largest hyperplane after the first $k$ is $\le k(m - 1)$. This means that removing a single $H_i$ and attempting to cover $\ge |H_i| - 1$ elements with smaller hyperplanes would require $\ge \frac{|H_i|}{k(m - 1)}$ new hyperplanes.  Since $\frac{\min |H_i|}{k(m - 1)} \ge \frac{n}{Ck}$, the number of additional hyperplanes required is $\ge \frac{n}{Ck^2 (m - 1)}$. Removing more of the $H_i$ would give an even greater increase in the number of hyperplanes used. This means that the only cover of $\ge n - 1$ elements of $[n]$ by $a \le k - 1 + \frac{n}{2Ck^2(m - 1)} < k - 1 + \frac{n}{Ck^2(m - 1)}$, hyperplanes is the cover of $[n]$ by $H_1, \ldots, H_k$. Thus, $MCB(a)$ must be satisfied by $M$ for $a \le  k - 1 + \frac{n}{2Ck^2(m - 1)}$. 
			
			\color{black}
			
			\item We can use similar reasoning as in Part 1. In general, removing $\ell$ of the $H_i$ and replacing them with hyperplanes \emph{not} belonging to the $H_1, \ldots, H_k$ uses up $\ge k - \ell + \ell \frac{\min |H_i|}{k(m - 1)}$ hyperplanes since the remaining hyperplanes have size $\le k(m - 1)$. Note that this lower bound increases as $\ell$ increases since $\min |H_i| >> k(m - 1)$. Setting $\ell = 1$ gives the lower bounds and this is the reflect 
		\end{enumerate}
	\end{proof}
	
	For the paving matroids considered in Proposition \ref{pavhyperplrat}, the Chow ring has a precise relation to the minimal degree $a$ such that the paving matroids satisfy $MCB(a)$. Note that the Chow ring of a matroid is equal to the Chow ring of an actual toric variety constructed out of a fan (p. 5 of \cite{BES}). 
	
	\begin{cor} \label{chowpav}
		Consider paving matroids on $E$ of rank $m + 1$ where the maximal hyperplanes are much larger than $\frac{n}{m}$ and don't vary much (in the sense described in Part 1 of Theorem \ref{pavhyperplrat}). For example, this includes paving matroids where the hyperplanes are given by very large blocks partitioning $E$ and the rest of the hyperplanes given by sets of size $m$. Then, the minimal degree $a$ where $MCB(a)$ is satisfied and the upper bound in Part 1 of Thoerem \ref{pavhyperplrat} decreases with the as the dimension of the quotients by the annihilators of each $x_{H_i}$ in the Chow ring $A^*(M)$ of $M$ increase.
	\end{cor}
	
	\begin{proof}
		This is an application of the formula \[ H(D(\mathcal{L}, t) = 1 + \sum_r  \left( \prod_{i = 1}^{k(r)} \frac{t(1 - t^{r_i - r_{i - 1} - 1})} {1 - t}   \right) f_{\mathcal{L}}(r)  \] for the Hilbert series of the Chow ring of the matroid on p. 526 of \cite{FY} after Corollary 2 on p. 525 of \cite{FY}, where $\mathcal{L}$ denotes the lattice of flats and $D(\mathcal{L})$ denotes the Chow ring construction, $r = (0 = r_0 < r_1 < \cdots < r_k \le \rk \mathcal{L})$ gives rank sequences of flags of flats, and $k = k(r)$ is the length of the rank sequence. We can remove the $f_{\mathcal{L}}(r)$ term involving the number of flags with a given rank sequence $r$ if we index over flags instead of rank sequences $r$. \\
		
		If we index the formula by flags of flats instead of indices themselves, we can see that the degree $k$ term of the Hilbert series corresponds to the number of flags with the given ranks. Note that the only flags affected by the $MCB(a)$ are those where which end with a hyperplane or the ground set itself (i.e. $r_k = r$ or $r_k = r - 1$, where $r = m + 1$). This increases with the sizes of the maximal hyperplane since this increases the number of smaller rank objects. In other words, the length $\ell$ ending with $[n]$ correspond to those of length $\ell$ ending with some flat of smaller rank. While the degrees of the variables considered stay the same, the change comes from the number of possible variables to consider (which correspond to possible flags of flats using the given ranks). We fix the degree and look for flats with given differences in ranks. There are more flats of rank $\le m - 1$ to substitute in. Note that the behavior entirely depends on those of the hyperplanes since \emph{any} subset of $E$ of size $\le m - 1$ has rank equal to its size. This implies that the number of chains of  hyperplane or a hyperplane and the ground set $[n]$ entirely depends on the size of the given hyperplane. Since the $x_{F_1}^{\alpha_1} \cdots x_{F_\ell}^{\alpha_\ell}$ from flags of flats $F_1 < \cdots < F_\ell$ and $\alpha_i$ such that $1 \le \alpha_i \le \rk F_{i + 1} - \rk F_i$ and $\sum \alpha_i = k$ form a basis for $A^k(M)$ as a vector space (p. 526 of \cite{FY}, Corollary 3.3.3 on p. 18 of \cite{BES}), the degree is given by $(m + 1) - 1 - \rk F_1$ if $F_\ell$ is a hyperplane of $M$.  \\

		Given a hyperplane $H_i$, the condition that $x_{H_i}$ is \emph{not} an annihilator is equivalent to stating that the flats corresponding to the variables in the monomials involved are either strictly contained in $H_i$ or strictly contain $H_i$ by the definition of the Chow ring of a matroid. This restricts the flags under consideration to a collection of flats contained in $H_i$, one ending at $H_i$, or one ending with $H_i$ and the ground set $[n]$. The analysis above then implies the conclusion after applying the arguments above. 
		
	\end{proof}
	
	\color{black}

	\begin{rem} \label{degintgen} ~\\
		\vspace{-3mm}
		\begin{enumerate}
			\item When the minimal degree $a$ for $MCB(a)$ decrases, the increase in coefficient size can be interpreted in terms of ``local'' complexity of the Chow ring of toric varieties built out of the Bergman fan of the matroids (p. 5 of \cite{BES}, Proposition 7.13 and Definition 7.14 on p. 431 -- 432 of \cite{AHK}).

			\item The general relation between having a low degree $a$ for the minimal degree such that the matroidal Cayley--Bacharach property $MCB(a)$ is satisfied and sizes of coefficients of the Hilbert series of the Chow ring also seems to apply to the case of arbitrary matroids when $a$ is small. The main difference appears to be in characterizing which collections of subsets can actually appear as hyperplanes of some matroid. There has been previous work studying such questions about possible subsets (e.g. \cite{YE}. \cite{De}). 
		\end{enumerate}
	\end{rem}
	
	\subsection{Supersolvable arrangements}
	
	Given an arrangement $\mathcal{A} = \{ H_1, \ldots, H_n \}$ of hyperplanes in $k^d$ for some field $k$, the associated matroid $M_{\mathcal{A}}$ has flats built out of intersections of the hyperplanes involved. In particular, the rank function is defined as $r(B) = n - \dim \bigcap_{i \in B} H_i$ and the flats are given by maximal sets of indices corresponding to intersections of hyperplanes equal to a particular linear subspace (formed by intersections of hyperplanes). We study $a$ such that these matroids $M_{\mathcal{A}}$ satisfy $MCB(a)$ when $n >> a^3$ and $\mathcal{A}$ is a line arrangement. This is in addition to a general description of $MCB(a)$ for $M_{\mathcal{A}}$ for arbitrary hyperplane arrangements $\mathcal{A}$ (Proposition \ref{degnumlinemcb}). In addition, we show that the ``nontrivial'' supersolvable line arrangements give a family of line arrangements where the number of possible degrees of unexpected curves decreases as the minimal $a$ such that $MCB(a)$ is satisfied increases (Proposition \ref{supersolmcbexp}). Finally, we end with some comments to topological properties of the arrangements in Remark \ref{topprophint}.  \\

	While the degree $a$ matroidal Cayley-Bacharach property $MCB(a)$ is defined as $\bigcup_{i = 1}^a F_i \supset E \setminus p \Longrightarrow \bigcup_{i = 1}^a F_i = E$ for any $p \in E = [n] = \{ 1, \ldots, n \}$, this can be rephrased in a simple way for hyperplane arrangements.

	\begin{prop} \label{degnumlinemcb} ~\\ 
		\vspace{-3mm}
		\begin{enumerate}
			\item Suppose that $a^2 \ll \frac{n}{a}$. Then, the matroid $M_{\mathcal{L}}$ associated to an arrangement $\mathcal{L}$ of $n$ lines with $a$ points of degree close to $\frac{n}{a}$ satisfies $MCB(a)$. Note that a collection of such ``high degree points'' is necessary in order for $MCB(a)$ to be satisfied nontrivially.  In general, $MCB(a)$ is satisfied when $a$ is very small compared to the number of lines or maximum multiplicity. 	
			
			\item In general, a matroid $M_{\mathcal{A}}$ built out of a hyperplane arrangement $\mathcal{A} = \{ H_1, \ldots, H_n \}$ satisfies $MCB(a)$ if and only if $a$ is less than or equal to the minimal number of intersections of elements of $\mathcal{A}$ (i.e. elements of the intersection lattice or intersection points in the case of a line arrangement) such that the indices cover $\mathcal{A} \setminus H_i$ for some $i \in [n] = \{ 1, \ldots, n \}$. By intersections of elements, we mean linear subspaces of the form $\bigcap_{i \in B} H_i$ for some $B \subset [n]$ with such subspaces written using the largest possible such subset $B$ with respect to inclusion. 
		\end{enumerate}
	\end{prop}
	
	\begin{proof}
		\begin{enumerate}
			\item 
			In the generic case, we can start with a suitable collection of lines parametrized by subsets $F_i$ of the index set $[n]$ (each giving rise to distinct intersection points). Since any two lines only intersect at one point, the remaining points of intersection (not coming from the $F_i$) have multiplicity $\le a$. Suppose that $a^2 << \frac{n}{a}$ and that the $F_i$ are not far from being evenly distributed in size (at least much larger than $a^2$ treating the variables as functions of $a$). Then, $MCB(a)$ must be satisfied since using $a$ points of intersection not all coming from the $F_i$ will not be able to be used to cover the ground set $[n] = \{ 1, \ldots, n \}$ indexing the hyperplanes of the arrangement $\mathcal{A}$. In general, the fact that $t_2 + t_3 \ge k + t_5 + 2t_6 + 3t_7 + \ldots$ with $t_i$ equal to the number of intersection points of multiplicity/degree $i$ (Hirzebruch \cite{Hir}) implies that there are many more low degree points than high degree ones, which implies that $MCB(a)$ must be satisfied if $a$ is small in general. A similar argument can be repeated if we consider the case of hyperplane arrangements and sets parametrizing intersctions of hyperplanes (and linear subspaces in general). \\
			
			\item We need $a$ such that a collection of $a$ intersections of hyperplanes in $\mathcal{A}$ either use up all the hyperplanes or miss $\ge 2$ of them. In the case of line arrangements, this means that taking $\le a$ intersection points of lines either uses up all the lines or we are missing $\ge 2$ of the lines. \\
		\end{enumerate}
	\end{proof}

	The case of line arrangements yields further connections to degees of unexpected curves arising from supersolvable line arrangements. This makes use of the following thereom of Hanumanthu--Harbourne \cite{Han} on supersolvable line arrangements wth a given number of modular points and their connections to degrees of unexpected curves.
	
	\begin{thm} (Hanumanthu--Harbourne, p. 3 of \cite{Han}) \\
		
		Let $\mathcal{L}$ be a line arrangement (over any field) with a modular point (i.e. an intersection point connected to all other intersection points by a line in $\mathcal{L}$). 
		
		\begin{enumerate}
			\item If $\mathcal{L}$ is \emph{not} homogeneous, then either $\mathcal{L}$ is a near pencil or it has two modular points. If it has two modular points, then $\mathcal{L}$ consists of $a \ge 2$ lines through one modular point and $b > a$ lines throug the other one. This means that there are $a + b - 1$ lines in $\mathcal{L}$ and $(a - 1)(b - 1)$ intersection points of multiplicity $2$. 
			
			\item If $\mathcal{L}$ has a modular point of multiplicity $2$, then $\mathcal{L}$ is trivial.
			
			\item If $\mathcal{L}$ is complex and homogeneous (i.e. each intersection point has the same multiplicity/degree) with the maximum multiplicity $> 2$, there are $\le 4$ modular points. If there are $3$ or $4$ modular points, we have the following possiblities:
			
			\begin{itemize}
				\item If there are $4$ modular points, then there are 6 lines in $\mathcal{L}$, the common multiplicity is $m = 3$, $t_2 = 3$, $t_3 = 4$, and $t_k = 0$ otherwise. Up to a change of coordinates, $\mathcal{L}$ consists of the lines $x = 0, y = 0, z = 0, x - y = 0, x - z = 0,$ and $y - z = 0$. The intersection pattern is like that of an equilateral triangle and its angle bisectors. 
				
				\item If there are $3$ modular points, then the common multiplicity is $m > 3$ and up to change of coordinates, $\mathcal{L}$ consists of the lines defined by the linear factors of $xyz(x^{m - 2} - y^{m - 2})(x^{m - 2} - z^{m - 2})(y^{m - 2} - z^{m - 2})$. This means that there are $3(m - 1)$ lines, $t_2 = 3(m - 2)$, $t_3 = (m - 2)^2$, $t_m = 3$, and $t_k = 0$ otherwise.
				
			\end{itemize}
			
		\end{enumerate}
	\end{thm}
	
	While we will focus on the final case since it has the most interesting structure, we will also consider the non-homogeneous case.
	
	\begin{prop} \label{supersolmcbexp} ~\\
		\vspace{-3mm}
		\begin{enumerate}
			\item If $\mathcal{L}$ is a non-homogeneous supersolvable line arrangement and satisfies $MCB(a)$ for some $a$, then the corresponding matroid satisfies $MCB(a)$ if and only if $a \le \frac{A + B - 1}{2}$, where $A$ and $B$ are the degrees of the modular points.
			
			\item Given a homogeneous supersolvable line arrangement $\mathcal{L}$ with $3$ modular points, the minimal degree $a$ such that the matroid corresponding to $\mathcal{L}'$ satisfies $MCB(a)$ nontrivially decreases as the number of posisble degrees of unexpected curves increases.
		\end{enumerate}
	\end{prop}
	
	\begin{proof}
		\begin{enumerate}
			\item The theorem above implies that we either have a near-pencil or two modular points. In the first case, $MCB(a)$ cannot be satisfied for any $a$ since $MCB(1)$ is \emph{not} satisfied. This is because the failure of $MCB(a)$ implies the failure of $MCB(b)$ for any $b > a$. As for the case of two modular points, let $A$ and $B$ be the degrees of the modular points. Then, the conclusion follows from labeling the individual lines of the arrangement by pairs of the form $(i, j)$ with $1 \le i \le A$ and $1 \le j \le B$. We find the minimal number of pairs such that the coordinates $i$ and $j$ use up all the elements of $[A + B] = \{ 1, \ldots, A + B - 1 \}$.
			
			\item In the final case, note that the counts of the $t_i$ in the case of $3$ modular points comes from the fact that intersection points of lines of $\mathcal{L}$ which involve $2$ factors \emph{not} involving $xyz$ actually intersect at $3$ such factors. Checking for possible $a$ where $MCB(a)$ can be satisfied by $\mathcal{L}'$ with the linear factors $x, y, z$ of $xyz$ removed corresponds to the possible degrees of unexpected curves throguh points corresponding to the duals of lines of $\mathcal{L}$. More precisely, $\mathcal{L}'$ satisfies $MCB(a)$ for $a \le \frac{m}{3}$ and $D$ is a possible degree of an unexpected curve if and only if $m \le D \le n - m - 1$ (Theorem 3.8 on p. 173, p. 180 -- 181 of \cite{DMO}). This gives a negative correlation between $MCB(a)$ degrees $a$ for $\mathcal{L}'$ and  the number of possible degrees of unexpected curves arising from $\mathcal{L}$.
		\end{enumerate}
	\end{proof}

	\begin{rem} \label{topprophint} ~\\
		\vspace{-3mm}
		\begin{enumerate}
			\item Recall that a central arrangement of linear subspaces is one where all the intersection of all of the linear subspaces is nonempty. In the case of line arrangements, the number of indices covered by a collection of intersection points can be expressed by the number of regions the corresponding central subarrangement splits the plane into. This can be expressed as the a specilization (substituting $t = -1$ into the variable) of the characteristic polynomial (Theorem 4.1 on p. 7 of \cite{Ar}) of the matroid corresponding to the central subarrangement. Using an inclusion-exclusion argument, the number of elements covered by a collection of intersection points can be bounded above by the sum of specializations of characteristic polynomials of matroids associated to central line arrangements.
			
			\item The arguments of Part 1 also apply in the case of hyperplane arrangements.
			
			\item Using the lattice of flats while representing each flat by a single point and connecting two points by a line if one flat is contained in the other, the $MCB(a)$ condition can be phrased in a graph-theoretic manner. It means that a collection of points connected to all but possibly one point $i \in [n]$ is connected to every point of $[n]$.  
		\end{enumerate}
	\end{rem}

	We continue to analyze supersolvable arrangements, but move from lines to the more general setting of hyperplanes. As in \cite{BEZ}, most of the arrangements considered will be assumed to be \emph{central}. These hyperplane arrangements give a clear connection between the lattice of flats of the associated matroid (i.e. lattice formed by intersections of hyperplanes) and the connected components/regions bounded by the collection of hyperplanes (called chambers). \\

	\begin{defn} \label{supersoldefs} (Definition 4.1 and Definition 4.2 on p. 273 -- 274 of \cite{BEZ}) \\
		\begin{enumerate}
			\item Writing $d$ for the rank, a \textbf{supersolvable geometric lattice} is defined as one having a maximal chain of form $\widehat{0} = V_0 \prec V_1 \prec \cdots \prec V_{d - 1} \prec V_d = \widehat{1}$, where $\widehat{0}$ and $\widehat{1}$ are minimal and maximal elements of the lattice (p. 273 of \cite{BEZ}) and $x \prec y$ means that $x < y$ and $x < z \le y \Longrightarrow z = y$. In our case, we take the elements of the lattice to be intersections of the hyperplanes of the arrangement and the ordering is given by reverse inclusion.
			
			\item A central arrangement $\mathcal{A}$ is \textbf{supersolvable} if its lattice $L(\mathcal{A})$ of intersections is a supersolvable lattice. \\
			
			\item For $1 \le i \le d$, let $e_i$ be the number of atoms of $L = L(\mathcal{A})$ that lie below $V_i$, but not $V_{i - 1}$. We have $e_1 = 1$ and $\sum_{i = 1}^d e_i$ is the number of atoms in $L(\mathcal{A})$. Also, the characteristic polynomial of $L$ is $\chi(L, t) = \prod_{i = 1}^d (t - e_i)$.
		\end{enumerate}
	\end{defn}
	
	\color{black}
	One of the three initial examples considered in \cite{BEZ} is the graph hyperplane. We consider the computations in more detail below.
	
	\begin{exmp} (Matroids of graph hyperplane arrangements and $MCB(a)$) \\
		Given a graph $G$ with vertex set $[n] = \{ 1, \ldots, n \}$, consider the hyperplane arrangement $\mathcal{A}_G$ formed by hyperplanes of the form $x_i = x_j$ for each $(i, j) \in E(G)$ (i.e. pairs forming an edge of $G$). Intersections of hyperplanes that are considered are of the form $x_{i_1} = \cdots  = x_{i_k}$ for some set $\{ i_1, \ldots, i_k \}$. Since flats consist of maximal collections of hyperplanes from the arrangement considered ($\mathcal{A}_G$ in this case) giving rise to a specific linear subspace, the flats of the matroid $M_{\mathcal{A}_G}$ associated to $\mathcal{A}_G$ has ground set given by the elements of $E(G)$ (edges of $G$) and the flats are $E(G|_{V_i})$, where $V_i \subset [n]$ and $G|_{V_i}$ is the restriction of $G$ to the vertex subset $V_i$. \\
		
		In this particular setting, checking whether $MCB(a)$ can be satisfied doesn't seem to depend on the degree $a$. 
		
		\begin{prop}
			The matroid $M_{\mathcal{A}_G}$ of a hyperplane arrangement $\mathcal{A}_G$ in $\mathbb{R}^n$ associated to a graph $G$ with vertex set $[n]$ satisfies $MCB(a)$ for some $a$ if and only if every edge is bounded by vertices of degree $\ge 2$.
		\end{prop}
		
		\begin{proof}
			To see this, we look at what happens when we omit a specific edge from the union of edges coming from some collection of flats (which can be taken to be $a$). Note that the flats come from edges inside the restriction of the graph $G$ to some subset of the vertex set $[n] = \{ 1, \ldots, n \}$. We can split into cases according to the degrees of the vertices bounding the missing edge $e$.
			
			\begin{enumerate}
				\item \textbf{Case 1:} There is an edge $e$ where each bounding vertex has degree $1$.
				
				In this case, it doesn't seem like $MCB(a)$ is satisfied for \emph{any} $a$. This is because the flats $F_i = E(G|_{V_i})$ can be taken to come from any collection of vertex sets $V_i$ with union equal to  $A \setminus \partial e$, where $A \subset [n]$ is the set of vertices of degree $\ge 1$ and $\partial e$ denotes the pair of vertices bounding the missing edge $e$. This would contain all the edges except $e$. Note that this case would be omitted if the graph $G$ is assumed to be connected.
				
				\item \textbf{Case 2:} There is an edge $e$ where one bounding vertex has degree $1$ and the other has degree $\ge 2$.
				
				The matroid $M_{\mathcal{A}_G}$ still does \emph{not} satisfy $MCB(a)$ in this case. This is because the vertex subsets $V_i \subset [n]$ can be taken to have union equal to $A \setminus p$, where $A$ is defined in the same way as in Case 1 and $p$ is the vertex in $\partial e$ of degree $1$. In that case, the restriction to the given set of vertices is still missing the edge $e$ but contains all others.

				\item \textbf{Case 3:} Each edge $e$ is bounded by vertices of degree $\ge 2$.
				
				In this case, the matroids $M_{\mathcal{A}_G}$ \emph{do} satisfy $MCB(a)$ regardless of the choice of $a$. By including the edges connected to each of the two vertices in $\partial e = \{ p, q \}$, any edges induced by restriction to a subset of the vertices $[n]$ including edges other than $e$ containing $p$ or $q$ in the boundary must include $p$ and $q$ as well. Thus, a collection of edges coming from restrictions of vertex sets misisng at most one edge of $G$ contains all of the edges of $G$.
			\end{enumerate}
		\end{proof}
		
	\end{exmp}

	\begin{rem} ~\\
		\vspace{-3mm}
		\begin{enumerate}
			
			\item A graphic arrangement is supersolvable if and only if the graph in question is \emph{chordal} (i.e. for any cycle with $\ge 4$ vertices, there is an edge of $G$ connecting two vertices which are not adjacent in the cycle -- see Remark 2.5 on p. 9 of \cite{Bi}).

			\item Some other examples to consider are polytopal arrangements from hyperplanes built out of facets of polytopes and Coxeter arrangements from finite subsets of $GL_d(\mathbb{R})$ (orthogonal reflections through hyperplanes) (p. 268 -- 269 of \cite{BEZ}).
		\end{enumerate}
	\end{rem}

	In general, the computation above and the definition of $MCB(a)$ for matroids $M_{\mathcal{A}}$ associated to hyperplane arrangements $\mathcal{A}$ seems to indicate some kind of forced connectivity since a ``missing hyperplane'' must intersect collections of intersections of other hyperplanes in some way. One way to do this would be to impose a dependency on the hyperplane intersections depending on the indices considered. However, we still need to check whether such a condition is necessary. \\
	
	For supersolvable arrangements, checking $MCB(a)$ can ``generically'' be reduced to a question on a smaller hyperplane arrangement.

	\begin{thm} \label{supersoldec} (Bj\"orner--Edelman--Ziegler, Theorem 4.3 on p. 274 of \cite{BEZ}) \\
		Every arrangement $\mathcal{A}$ of $\rk \le 2$ is supersolvable. An arrangement $\mathcal{A}$ of rank $d \ge 3$ is supersolvable if and only if $\mathcal{A} = \mathcal{A}_0 \sqcup \mathcal{A}_1$, where $\mathcal{A}_0$ is a supersolvable arrangement of rank $d - 1$ and, for any $H', H'' \in \mathcal{A}_1$ with $H' \ne H''$, there is an $H \in \mathcal{A}_0$ such that $H' \cap H'' \subset H$.
	\end{thm}
	
	Using this result, we can make the following observations.

	\begin{prop} \label{supersolmcb}
		Let $\mathcal{A}$ be a central supersolvable hyperplane arrangement.
		\begin{enumerate}
			\item Writing $\mathcal{A} - \mathcal{A}_0 \sqcup \mathcal{A}_1$ as in Theorem \ref{supersoldec}, let $d = \rk \mathcal{A}$ and $\mathcal{B}_0 = \{ V_{d - 1} \vee (H' \wedge H'') : H', H'' \in \mathcal{A}_1 \} $ be the hyperplanes in $\mathcal{A}_0$ containing the pairwise intersections of elements of $\mathcal{A}_0$. Given  a collection of hyperplane intersections/flats $P$, write $P = P_0 \sqcup P_1$ with $P_0$ only from hyperplanes in $\mathcal{A}_0$ and $P_1$ involving hyperplanes from $\mathcal{A}_1$ in each intersection. \\
			
			In this setting, $M_{\mathcal{A}}$ satisfies $MCB(a)$ if and only if the following conditions hold each collection $P = P_0 \sqcup P_1$ ($k := |P_1|$) using up $\ge |\mathcal{A}| - 1$ hyperplanes and $1 \le k \le a$:

				\begin{itemize}
					
					\item Let $BP_0$ be the counterpart of $\mathcal{B}_0$ for $P$ built out pairwise intersections of elements of $P_1$. $M_{P_1}$ satisfies $MCB(k)$ and the $\le a - k$ hyperplanes in  $BP_0$ use up all of the hyperplanes in $\mathcal{A}_0$.

					\item $M_{BP_0}$ satisfies $MCB(a - k)$ and $P_1$ uses up all the hyperplanes in $\mathcal{A}_1$.
				\end{itemize}
			
			In particular, it suffices to have $M_{\mathcal{A}_1}$ satisfy $MCB(k)$ and $M_{\mathcal{A}_0 \setminus \mathcal{B}_0}$ satisfy $MCB(k)$ and $M_{\mathcal{A}_0 \setminus \mathcal{B}_0}$ for each $0 \le k \le a$. 
			
			\item The central supersolvable hyperplane arrangements such that $MCB(d)$ is satisfied for the minimal nontrivial degree $a$ can take \emph{any} possible characteristic polynomial or rank generating function. This means that any central supersolvable hyperplane arrangement of rank $d$ has the same characteristic polynomial as one satisfying $MCB(d)$.

			\item Let $u = \rk \mathcal{A}$ and $\Omega_u$ be the intersection of all the hyperplanes in $\mathcal{A}$, and $\Omega_{u - 1}$ be the intersection of the hyperplanes in $\mathcal{A}_0$. For any $R \in \mathcal{A}_1$, the intersection $R \cap \Omega_{u - 1} = \Omega_u$. In particular, this implies that any pair of flats of $M_{\mathcal{A}}$ where one of them is the ground set $\mathcal{A}_0$ of $M_{\mathcal{A}_0}$ and the other contains an element of $\mathcal{A}_1$ covers the entire ground set of $M_{\mathcal{A}}$.

		\end{enumerate}
	\end{prop}
	
	\begin{proof}
		\begin{enumerate}
			\item  Consider a collection of intersections of $a$ hyperplanes of $\mathcal{A}$ which is ``missing'' at most hyperplane. Let $P$ be a colection of such hyperplane intersections with $P_0$ only involving hyperplanes in $\mathcal{A}_0$ and $P_1$ involving hyperplanes in $\mathcal{A}_1$ (and possibly hyperplanes in $\mathcal{A}_0$). We can partition the cases involved into ones where $|P_1| = k$ as $k$ varies over $0 \le k \le a$. This potential missing hyperplane is either in $\mathcal{A}_1$ or $\mathcal{A}_0$. If we start indexing the hyperplane intersections by ones that involve elements of $\mathcal{A}_1$, the intersections involved induce a collection of intersections of elements of $\mathcal{A}_1$. These intersections must also include the (unique) hyperplanes in $\mathcal{A}_0$ contain pairwise intersections of hyperplanes in $P_1$. Omitting these from the elements of $\mathcal{A}_1$, the remaining $a - k$ hyperplane intersections (from $P_0$) form $BP_0$. If the potential missing element is in $\mathcal{A}_1$, the elements of $BP_0$ use up all the elements of $\mathcal{A}_0$. In order for $MCB(a)$ to be satisfied, the missing element in $\mathcal{A}_1$ should actually be covered by $P_0$. This is the statement that $M_{P_1}$ satisfies $MCB(k)$. If the potential missing element is in $\mathcal{A}_0$, we have that $P_1$ uses up all the hyperplanes in $\mathcal{A}_1$. This means that the elements of $BP_0$ satisfy $MCB(a - k)$ as we already have a cover.

			\item Note that $e_d = |\mathcal{A}_1|$ (p. 275 of \cite{BEZ}). If $\mathcal{A}_1$ is a pencil of hyperplanes containing a single $(d - 2)$ linear subspace of some fixed $H \in \mathcal{A}_0$ which do \emph{not} contain the line $V_{d - 1}$, then intersecting any two of the hyperplanes in $\mathcal{A}_1$ means intersecting \emph{all} of the hyperplanes in $\mathcal{A}_1$. This means that any collection of intersections of hyperplanes in $\mathcal{A}$ which where at most $1$ hyperplane is not involved actually involves \emph{all} of the hyperplanes in $\mathcal{A}$ and the $MCB(a)$ property is satisfied for any $a$ such that this question is nontrivial. This can be done at each step of the construction of a supersolvable hyperplane arrangement of rank $\ge 3$. For the base case of a rank 2 supersolvable hyperplane arrangement, there are no restrictions on the ``base'' characteristic polynomial since \emph{any} central hyperplane arrangement of rank $\le 2$ is supersolvable by Theorem \ref{supersoldec}. The conclusion follows from noting that $\chi(L, t) = \prod_{i = 1}^d (t - e_i)$ (Part 3 of Definition \ref{supersoldefs}).

			\item In general, $V_{d - 1}$ can be taken to be a line contained in the common intersection $\Omega_{d - 1}$ of the hyperplanes in $\mathcal{A}_0$. Given a central hyperplane arrangement of rank $u$, let $\Omega_u$ be the intersection of all the hyperplanes in the arrangement. The new hyperplanes $A_i \in \mathcal{A}_1$ are those do \emph{not} contain $V_{d - 1}$. Choosing an initial such hyperplane $A_1$ to put in $\mathcal{A}_1$, we actually have that $\Omega_u = \Omega_{u - 1} \cap A_1$. Since $A_1 \not\supset V_{d - 1}$, we have that  $A_1 \not\supset \Omega_{u - 1}$ and $\dim A_1 \cap \Omega_{u - 1} = d - u + 1 - 1 = d - u$. Since $A_1 \cap \Omega_{u - 1}$ contains the intersection of \emph{all} the hyperplanes in the arrangement although it is of the same dimension (due to the rank), we have that $\Omega_u = \Omega_{u - 1} \cap A_1$. The remaining choices involve which $(d - 2)$-planes to use for the intersections of pairs of elements of $\mathcal{A}_1$ and what hyperplanes to place in them. Since the $(d - 2)$-planes must contain $\Omega_u$ (which is $\Omega_3$ in this case), the $(d - 2)$-planes depend on a choice of $d - 2 - (d - u) = u - 2$-planes (which are lines in this case). \\

		\end{enumerate}
	\end{proof}

	\color{black}

	\color{black}


\begin{thebibliography}{8}
		\bibitem{AHK} K. Adiprasito, J. Huh, and E. Katz, Hodge theory for combinatorial geometries, Annals of Mathematics 188(2) (2018), 381 -- 432.
		
		
		
		\bibitem{Ar} F. Ardila, Tutte polynomials of hyperplane arrangements and the finite field methods, to appear in Handbook of the Tutte Polynomial and Related Topics (edited by J. A. Ellis-Monaghan and I. Moffatt), Chapman and Hall/CRC (2022), \url{http://math.sfsu.edu/federico/Articles/Tuttehyparr.pdf}
		
		\bibitem{ABD} F. Ardila, C. Benedetti, and J. Doker, Matroid Polytopes and their Volumes, Discrete \& Computational Geometry 43 (2010), 841 -- 854.
		
		
		
		
		\bibitem{Bi} C. Bibby, Abelian arrangements, University of Oregon PhD thesis (2015).
		
		\bibitem{BES} S. Backman, C. Eur, and C. Simpson, Simplicial generation of Chow rings of matroids, FPSAC 2020 (Online), S\'eminaire Lotharingen de Combinatoire 84B, Article \# 52, 1 -- 11 (2020).
		
		\bibitem{BEZ} A. Bj\"orner, P. H. Edelman, and G. M. Ziegler, Hyperplane arrangements and a lattice of regions, Discrete \& Computational Geometry 5 (1990), 263 -- 288.
		
	
		
		\bibitem{CHMN} D. Cook II, B. Harbourne, J. Migliore, and U. Nagel, Line arrangements and configurations of points with an unexpected geometric property, Compositio Mathematica 154 (10) (2018), 2150 -- 2194.
		
		\bibitem{De} M. Deza, Perfect Matroid Designs, Ch. 2 of ``Matroid Applications'' 40 edited by N. White (1992), 54 -- 72.
		
		\bibitem{DMO} M. Di Marca, G. Malar, and A. Oneto, Unexpected curves arising from special line arrangements, Journal of Algebraic Combinatorics 51(2) (2020), 171 -- 194. 
		
		\bibitem{Dl} M. Dlugosch, New light on Bergman complexes by decomposing matroid types, FPSAC 2012, Nagoya, Japan, DMTCS proc. AR (2012), 181 -- 190. 
		
	
		\bibitem{E} N. Y. Erokhovets, Gal's conjecture for nestohedra corresponding to complete bipartite graphs, Proceedings of the Steklov Institute of Mathematics 266 (2009), 120 -- 132. 
		
		\bibitem{FS} E. M. Feichtner and B. Sturmfels, Matroid polytopes, nested sets and Bergman fans, Portugaliae Mathematica 62(4) (2005), 437 -- 468.
		
		\bibitem{FY} E. M. Feichtner and S. Yuzvinksy, Chow rings of toric varieties defined by atomic lattices, Inventiones mathematicae 155 (2004), 515 -- 536.
		
		
		
		\bibitem{Han} K. Hanumanthu and B. Harbourne, Real and complex supersolvable line arrangements in the projective plane, Journal of Algebraic Combinatorics (2020), 1 -- 19.
		
		\bibitem{HMNT} B. Harbourne, J. Migliore, U. Nagel, and Z. Teitler, Unexpected hypersurfaces and where to find them, Michigan Mathematical Journal 70(2) (2021), 301 -- 339.
		
		\bibitem{Hir} F. Hirzebruch, Arrangements of lines and algebraic surfaces, Arithmetic and geometry, Birkh\"auser, Boston, MA (1983), 301 -- 339.
		
	
		
		
		\bibitem{LU} J. Levinson and B. Ullery, A Cayley-Bacharach theorem and plane configurations, to appear in Proceedings of the American Mathematical Society \url{https://arxiv.org/pdf/2102.08525.pdf}
		
		
		
	
		
		\bibitem{Ox} J. G. Oxley, Matroid Theory, Oxford University Press (2006).
		
		\bibitem{Pa} S. Park, Matroids satisfying the matroidal Cayley--Bacharach property and ranks of covering flats, \url{https://arxiv.org/pdf/2203.14953.pdf}
		
		\bibitem{Pv} R. Pendavingh and J. van der Pol, On the number of matroids compared to the number of sparse paving matroids, The Electronic Journal of Combinatorics 22(2) (2015), 1 -- 17. 
		
		
		
		\bibitem{Pos} A. Postnikov, Permutohedra, Associahedra, and Beyond, International Mathematics Research Notices 6 (2009), 1026 -- 1106.
		
		\bibitem{V} V. D. Volodin, Cubical realizations of flag nestohedra and proof of Gal's conjecture for them, Communications of the Moscow Mathematical Society 65(188) (2010), 188 -- 190.
		
	
		
		\bibitem{YE} P. Young and J. Edmonds, Matroid Designs, Journal of Research of the National Bureau of Standards -- B. Mathematical Sciences 778, No. 1 \& 2, January -- June (1973).
		
		
	\end{thebibliography}
\end{document}